\documentclass{amsart}
\usepackage[textsize=scriptsize]{todonotes}

\usepackage{etoolbox}
\usepackage[margin=1in,marginparwidth=0.8in]{geometry}
\usepackage{comment}

\usepackage{amsmath}
\usepackage{amssymb}
\usepackage{amsthm}
\usepackage{amscd}
\usepackage{enumerate}
\usepackage[pdfusetitle,unicode,hidelinks]{hyperref}
\usepackage{bbm}
\usepackage{etoolbox}

\usepackage[utf8]{inputenc}
\usepackage[T1]{fontenc}

\newtheorem{proposition}{Proposition}
\newtheorem{corollary}[proposition]{Corollary}
\newtheorem{lemma}[proposition]{Lemma}
\newtheorem{theorem}[proposition]{Theorem}

\newtheorem*{conjecture*}{Conjecture}
\newtheorem*{theorem*}{Theorem}
\newtheorem*{corollary*}{Corollary}
\newtheorem*{proposition*}{Proposition}
\newtheorem*{lemma*}{Lemma}
\theoremstyle{definition}

\newtheorem*{definition*}{Definition}
\newtheorem*{construction*}{Construction}
\theoremstyle{remark}
\newtheorem{remark}[proposition]{Remark}
\newtheorem*{remark*}{Remark}

\newtheorem*{example*}{Example}

\let\scr=\mathcal
\let\bb=\mathbb

\def\A{\bb A}

\newcommand{\SH}{\mathcal{SH}}

\let\lim=\relax
\DeclareMathOperator*{\lim}{lim}

\def\Fun{\mathrm{Fun}}

\newcommand{\wequi}{\simeq}

\DeclareRobustCommand{\ul}{\underline}

\def\op{\mathrm{op}}

\let\cat=\mathrm
\def\Sm{{\cat{S}\mathrm{m}}}

\def\Nis{\mathrm{Nis}}

\def\mot{\mathrm{mot}}

\usepackage[textsize=scriptsize]{todonotes}

\numberwithin{proposition}{section}
\numberwithin{equation}{section}

\newtoggle{final}
\toggletrue{final}

\iftoggle{final} {
\renewcommand{\todo}[1]{}
\newcommand{\NB}[1]{}
}{ % else

\newcommand{\NB}[1]{\todo[color=gray!40]{#1}}
}

\newcommand{\Shv}{\mathcal{S}\mathrm{hv}}

\title{A $C_2$-equivariant Gabber presentation lemma}
\author{Tom Bachmann}

\begin{document}
\maketitle

\section{Introduction}
In \cite[\S3]{gabber-presentation-lemma}, O. Gabber established a certain technical ``presentation lemma'' which he used to prove exactness of so-called Gersten complexes.
This innocuous looking statement has since become a cornerstone of motivic homotopy theory over fields \cite{A1-alg-top}.
Much energy has been expended into generalizing it---it is now known also over finite fields \cite{gabber-lemma-finite-fields} and, with an additional ``shift'', also over much more general bases \cite{schmidt2018stable,gabber-lemma-noetherian-domains}.

In this short note we establish a variant of Gabber's lemma in the setting of varieties with an action by the group $C_2$ of order two.
From this we deduce a modified form of Gersten injectivity.\footnote{It is well-known that Gersten injectivity in its usual form does not hold equivariantly, see e.g. \cite[Remark 6.3]{rigidity-equiv-pseudopretheories}.}

Using these results, we extend some of the standard properties of motivic homotopy theory over a field to the $C_2$-equivariant setting (stable connectivity, homotopy $t$-structure).
We only give some of the most immediate applications, leaving many extensions to future work.
This is because I view the results as somewhat ``experimental'', ``a first foray'' into equivariant Gabber lemmas.
Similar results should be provable for many other groups, and the good consequences for motivic homotopy theory should then follow for all these groups.

I leave it to someone with more geometric clout than myself to tackle these issues.

%\subsection*{Future work}
%It seems likely that our results can be extended in various ways (other groups, further applications, etc.).
%This will be explored elsewhere.

\subsection*{Notation}
We denote by $\A^\sigma$ the affine line with $C_2$-action given by the sign representation, and also by $\sigma \in C_2$ the generator.

\subsection*{Acknowledgements}
I would like to thank Marc Levine and Charanya Ravi for very helpful discussions.

\section{Presentation lemma}
Our main result is the following.
It is a generalization in both statement and proof of \cite[Lemma 3.1, Remark 3.7]{gabber-presentation-lemma}.
\begin{comment}
\begin{theorem} \label{thm:main-old}
Let $k$ be an infinite field of characteristic $\ne 2$.
Let $X$ be a smooth $k$-scheme with an action by $C_2$.
Let $Z \subset X$ be a closed invariant subscheme of positive codimension and $z \in X^{C_2} \cap Z$ be a closed, rational point.
Then there exist:
\begin{itemize}
\item A $C_2$-equivariant étale neighborhood $X' \to X$ of $z$.
\item A smooth $C_2$-scheme $W$.
\item A $C_2$-equivariant étale neighborhood $X' \to \A^1_W$ of $Z_{X'}$.
  Here $\A^1$ carries either the trivial or the sign action.
  Moreover the composite $Z_{X'} \to \A^1_W \to W$ is finite.
\end{itemize}
\end{theorem}
\end{comment}

\begin{theorem} \label{thm:main}
Let $k$ be an infinite field of characteristic $\ne 2$.
Let $X$ be a smooth $k$-scheme with an action by $C_2$.
Let $Z \subset X$ be a closed invariant subscheme of positive codimension and $z \in Z$ be a point (not necessarily closed).
Then there exist:
\begin{itemize}
\item A $C_2$-equivariant étale neighborhood $X' \to X$ of $C_2z$.
\item A smooth $C_2$-scheme $W$.
\item A $C_2$-equivariant étale neighborhood $X' \to \A^1_W$ of $Z_{X'}$.
  Here $\A^1$ carries either the trivial or the sign action.
\end{itemize}

Moreover the following hold:
\begin{enumerate}
\item The composite $Z_{X'} \to \A^1_W \to W$ is finite.
\item If the action on $\A^1_W$ is the sign action, then $X' \to W$ admits a section.
\end{enumerate}
\end{theorem}

\begin{remark}
The existence of the lift of $W$ in the case of a non-trivial action on $\A^1$ is important in the applications.
Luckily this also naturally arises in the proof.
\end{remark}

\begin{remark}
To illustrate the challenges in proving Theorem \ref{thm:main}, I propose to discuss some special cases.
\begin{enumerate}
\item Suppose $Z = X^{C_2}$ is smooth of dimension $n$, codimension $1$ in $X$.
  Since $Z_{X'} \to W$ factors through $W^{C_2}$ and is finite, we find for dimension reasons that $W = W^{C_2}$ has trivial action.
  Since the action on $X$ is non-trivial, we see that $\A^1$ must carry the sign action in this case.
  Now $Z_{X'} \to \A^1_W$ factors through $(\A^1_W)^{C_2} = W$, and so we conclude that $Z_{X'} \to W$ is a clopen immersion.

  Gabber's proof of his lemma starts by constructing a map $X \to \A^n$ with finite restriction to $Z$, and then uses the remaining coordinate to convert the finite map into a closed immersion.
  This cannot work in our case (unless $Z$ is isomorphic to an open subset of $\A^n$).
  Our method is to instead pull back the map to $\A^n$ along $W \to \A^n$; this will yield an étale neighborhood of $z$ since $z \in W$.

\item Now suppose in contrast that $z \not\in X^{C_2}$.
  If we attempt to mimic the same strategy, we must find a smooth closed subscheme $W \subset X$ of codimension $1$ containing $z$.
  This is not possible in general unless $k$ is perfect (because $z$ may be a point of codimension $1$ with closure which is not generically smooth.)
  We must thus use a somewhat different argument, and construct a map to $\A^1_W$ where $\A^1$ has the trivial action.
\end{enumerate}
\end{remark}

We now proceed with the proof.
\begin{lemma} \label{lemm:functions}
Let $X$ be an affine finite type $k$-scheme, $k$ an infinite field.
Let $S \subset X$ be a finite set of closed points.
There exists a function $f: X \to \A^1$ inducing a universal injection $S \hookrightarrow \A^1$.

More generally, suppose given $T \subset X$ another finite set of closed points, disjoint from $S$.
Let $X \hookrightarrow \A^N$ be an embedding.
Suppose that after geometric base change, $T$ has size $n$.
Among the set of polynomial functions on $X$ of degree $\le n+1$ (with respect to the embedding $X \hookrightarrow \A^N$) vanishing on $T$, those which induce a universal injection $S \hookrightarrow \A^1 \setminus 0$ are dense.
\end{lemma}
\begin{proof}
The first claim follows from the second one by choosing an embedding $X \hookrightarrow \A^N$ and setting $T = \emptyset$ (and using that open dense subsets of affine spaces have points over $k$).

The desired condition (injectivity after geometric base change) is open.
To prove it can be satisfied, we may thus assume that $k$ is algebraically closed.
We must find a polynomial $P$ of degree $\le n+1$ vanishing on $T$ satisfying that: for every $s \in S$ we have $P(s) \ne 0$, and for $s' \ne s \in S$ we have $P(s) \ne P(s')$.
Each of these conditions is open, so we may satisfy them separately.
By temporarily adding $s'$ to $T$, it will suffice to show: there exists a polynomial $P$ of degree $\le |T|$ vanishing at $T$ but not at $s$.
Taking a product over terms for each $t \in T$, it suffices to find a linear polynomial vanishing at some specific $t \in T$, but not at $s$.
Since $t \ne s$ this is clearly possible.
\end{proof}

\begin{lemma} \label{lemm:functionsC2}
Let $X$ be an affine finite type $k$-scheme with a $C_2$-action, $k$ an infinite field of characteristic $\ne 2$.
Let $S \subset X$ be a finite set of closed points.
\begin{enumerate}
\item For $x \in X^{C_2} \setminus S$ there exists an equivariant map $f: X \to \A^1$ (where $\A^1$ has the trivial action) with $f(x) \not\in f(S)$.
\item If $S \subset X \setminus X^{C_2}$ then there exists an equivariant map $f: X \to \A^\sigma$ (where $\A^\sigma$ denotes $\A^1$ with the sign action) with $f(s) \ne 0$ for $s \in S$.
\item Suppose $S \subset X \setminus X^{C_2}$ and $T \subset X$ is another finite set of closed points, disjoint from $S \cup \sigma S$.
  There exists $f: X \to \A^\sigma$ vanishing on $T$ and non-vanishing on $S$.
\end{enumerate}
\end{lemma}
\begin{proof}
(1) Apply Lemma \ref{lemm:functions} to $X/C_2$ and $S/C_2 \cup \{x\}$.

(2) Apply Lemma \ref{lemm:functions} to $X$ and $C_2 S$ to obtain a (non-equivariant) function $f_0: X \to \A^1$.
We claim that $f=f_0 - f_0\sigma$ has the desired property, where $\sigma \in C_2$ denotes the generator.
By construction $f$ is an equivariant map to $\A^\sigma$.
To check that $f(s) \ne 0$ for $s \in S$ we may base change to the algebraic closure.
Now by assumption $s \ne \sigma s$, and hence $f(s) = f_0(s) - f_0(\sigma s)$ is non-vanishing by construction of $f_0$.

(3) Choose a (non-equivariant) embedding $X \hookrightarrow \A^N$ and consider the space of (non-equivariant) polynomials $P$ of degree $\le n$ such that $P$ vanishes on $T$.
Here $n$ is the sum of the geometric numbers of points in $S$ and $T$.
The subspace of polynomials such that $P-\sigma P$ does not vanish on $S$ is open; it will suffice to show it is non-empty.
This we may do after geometric base change.
Now $S \cup \sigma S = S' \amalg \sigma S'$.
Set $T' = T \cup \sigma S'$.
Lemma \ref{lemm:functions} yields $P$ vanishing on $T'$ and non-vanishing on $S'$.
This has the desired property.
\end{proof}

\begin{proof}[Proof of Theorem \ref{thm:main}.]
\textbf{First reductions.}
If the action of $C_2$ on $X$ is trivial (in some étale neighborhood of $z$), then we may appeal to the original Gabber lemma.
Hence we may assume that the action of $C_2$ is non-trivial.
Moreover, we may assume that $X$ is affine \cite[Remark 3.10]{hoyois-equivariant}.

We shall treat two cases, each in several steps.

\textbf{Case $z \in X^{C_2}$.}
Pick a closed specialization $\bar z \in X^{C_2}$ of $z$.
We can decompose the tangent space $T_{\bar z} X$ as a sum of trivial and sign representations, and then write $T_z X \wequi k(\bar z)^n \oplus k(\bar z)^{(m+1)\sigma}$.
If the action of $C_2$ is trivial on $T_{\bar z} X$, then it is trivial on an open neighborhood\footnote{Let $X = Spec(A)$ and $\bar z$ correspond to $m$. Write $A_\pm$ for the eigenspaces of the action. Then $A_+ \to A$ is finite. If the action on $m/m^2$ is trivial, then $(m_m)_-/((m_m)_+ (m_m)_-) = 0$, and hence $(m_m)_- = 0$ by Nakayama's lemma applied to the finite $(A_m)_+$-module $(m_m)_-$. It follows that $A_m = (A_m)_+$.} of $\bar z$, which we have excluded.
Thus $m \ge 0$.

\textbf{Step 1:} We claim that there exist maps $f_1, \dots, f_n: X \to \A^1$ and $g_1, \dots, g_m: X \to \A^\sigma$ such that if $(f,g): X \to \A^{n+m\sigma}$ denotes the induced map, then $(f,g)^{-1}((f,g)(\bar z)) \cap Z$ has dimension $0$.
We first find $f_1, \dots, f_n$ such that $\dim f^{-1}(f(\bar z)) \cap Z \le \dim Z - n$ and $\dim f^{-1}(f(\bar z)) \cap Z^{C_2} = 0$, where $f:X\to\A^n$ is the induced map.
Note that $\dim Z^{C_2} \le \dim X^{C_2} = n$.
Working inductively\NB{is this really legit?}\footnote{Observe that $(f,g)^{-1}((f,g)(\bar z)) \subset \cap_i f_i^{-1}(f_i(\bar z)) \cap \cap_j g_j^{-1}(g_j(\bar z))$.}, it will be enough to find $f_1$ such that $\dim f_1^{-1}(f_1(\bar z)) \cap Z < \dim Z$ and $\dim f_1^{-1}(f_1(\bar z)) \cap Z^{C_2} < \dim Z^{C_2}$.
Choosing one point in each component of positive dimension of $Z$ and $Z^{C_2}$ we can use Lemma \ref{lemm:functionsC2}(1) to find $f_1$ such that $f_1^{-1}(f_1(\bar z))$ does not completely contain any component of positive dimension of $Z$ or $Z^{C_2}$, which implies what we need.
Next note that $\dim Z \le \dim X -1 = n+m$, so $\dim f^{-1}(f(\bar z)) \cap Z \le m$.
Each component of $f^{-1}(f(\bar z)) \cap Z$ of positive dimension contains a point outside $Z^{C_2}$ (since $\dim Z(f_1, \dots, f_n) \cap Z^{C_2} = 0$).
Hence we can use Lemma \ref{lemm:functionsC2}(2) to find $g_1$ with $\dim (f,g_1)^{-1}((f,g_1)(\bar z)) \cap Z \le m-1$.
Iterating this procedure we obtain what we need.

\textbf{Step 2:} Let $p_1, \dots, p_N \in \scr O_X(X)$ generate $\scr O_X(X)$.
By symmetrizing and anti-symmetrizing, we can assume that $\sigma p_i = \pm p_i$, where $\sigma \in C_2$ denotes the generator (recall that $k$ has characteristic not $2$).
We further add the functions from step (1) to this set.
This way we obtain an equivariant closed immersion $X \hookrightarrow \A^{N + M\sigma}$.

\textbf{Step 3:} We claim that a general equivariant linear projection $\varphi: \A^{N + M\sigma} \to \A^{n+m\sigma} \times \A^\sigma$ has the following properties:
\begin{enumerate}[(a)]
\item The composite $Z \to \A^{N+M\sigma} \to \A^{n+m\sigma} \times \A^\sigma \to \A^{n+m\sigma}$ is quasi-finite at $\bar z$.
\item The composite $X \to \A^{N+M\sigma} \to \A^{n+m\sigma} \times \A^\sigma$ is étale at $\bar z$.
\end{enumerate}
Indeed both conditions are open\footnote{Note that spaces of equivariant maps are fixed points on spaces of all maps, so it suffices to show the openness ignoring equivariance. For étaleness see e.g. \cite[\S3.2.2]{kai-moving-lemma}. For the quasifiniteness, apply openness of the quasi-finite locus \cite[Tag 01TI]{stacks-project} to the universal projection $Z \times P \to \A^{n+m\sigma} \times P$.}, so we need only check that the sets are non-empty.
For (a), this holds because among the coordinates are the functions $f_1, \dots, f_n, g_1, \dots, g_m$ of step (1).
For (b), we may first base change to an algebraic closure of $k$.
Doing so splits $\bar z$ into finitely many points; it will suffice to show that being étale at each one is an open non-empty condition.
Thus we may assume that $\bar z$ is a rational point.
Now the claim follows using \cite[17.11.1]{EGAIV}, because the target is isomorphic to $T_z X$ and the linear inclusion $T_z X \to k^{N+M\sigma}$ equivariantly splits ($k$ having characteristic $\ne 2$).

\textbf{Step 4:} Choose a projection as in step (3) (this is possible since $k$ is infinite).
We obtain a map $\varphi: X \to \A^{n+m\sigma} \times \A^\sigma$ which is étale at $\bar z$ and such that $Z \to \A^{n+m\sigma}$ is quasi-finite at $\bar z$.
Since the étale and quasi-finite loci are open \cite[Tags 02GI and 01TI]{stacks-project}, shrinking $X$ we may assume $\varphi$ is étale and quasi-finite.
Let $W \subset X$ denote the vanishing locus of $g_{m+1}: X \to \A^\sigma$.
This is a closed subscheme étale over $\A^{n+m\sigma}$, whence smooth.
Moreover $X^{C_2} \subset W$, so $z \in W$.
Consider the pullback squares
\begin{equation*}
\begin{CD}
X_1 @>>> W \times \A^\sigma @>>> W \\
@VVV           @VVV                  @VVV \\
X @>>> \A^{n+m\sigma} \times \A^\sigma @>>> \A^{n+m\sigma}.
\end{CD}
\end{equation*}
We have $W \hookrightarrow X$ and so $W \times_{\A^{n+m\sigma}} W \hookrightarrow X_1$.
The diagonal $W \to W \times_{\A^{n+m\sigma}} W$ is open (both schemes being étale over $\A^{n+m\sigma}$) and closed ($W \to \A^{n+m\sigma}$ being separated); hence $W \times_{\A^{n+m\sigma}} W = W \amalg W'$.
Set $X_1' = X \setminus W'$ and consider the pullback squares
\begin{equation*}
\begin{CD}
W_z^h @>>> X_2 @>>> W_z^h \times \A^\sigma @>>> W_z^h \\
@VVV     @VVV         @VVV                      @VVV \\
W @>>> X_1' @>>> W \times \A^\sigma @>>> W.
\end{CD}
\end{equation*}
%The map $W \to X_1$ is the canonical lift (coming from $W \hookrightarrow X$).
Since $W_z^h$ lifts to $X_2$ as indicated, $z$ lifts to $X_2$.
Because $Z \to \A^{n+m\sigma}$ is quasi-finite at $z$ so is $Z_{X_2} \to W_z^h$.
We can write \cite[Tag 04GJ]{stacks-project} \[ Z_{X_2} = A_1 \amalg \dots \amalg A_n \amalg B\] with $A_i$ local and finite over $W_z^h$ and $B \to W_1$ not quasi-finite at any point over $z$.
It follows that $z \in A_i$ for some $i$, say $i=1$.
Set \[ X_3 = X_2 \setminus (\bigcup_{i >1} A_i \cup B). \]
Then $Z_{X_3} \to W_z^h$ is finite and the composite $Z_{X_3} \to X_3 \to W_z^h \times \A^\sigma$ is unramified and finite.
Since $Z_{X_3} (\wequi A_1)$ has fiber over $z$ consisting only of $z$, the fiber of $Z_{X_3} \to W_z^h \times \A^\sigma$ over $z$ is a closed immersion (use \cite[Tag 04XV]{stacks-project}), and hence so is $Z_{X_3} \to W_z^h \times \A^\sigma$ by Nakayama's lemma.

\textbf{Step 5:}
Note that for $i>1$, $W_z^h \cap A_i \subset W_z^h \subset X_2$ is closed but does not contain $z$, so must be empty.
Similarly for $W_z^h \cap B$.
It follows that $W_z^h \subset X_3$.
The map $Z_{X_3} \to \varphi^{-1}(\varphi(Z_{X_3}))$ is open (both schemes being étale over $\varphi(Z_{X_3}) \wequi Z_{X_3}$) and closed (both schemes being closed in $X_3$).
Hence $\varphi^{-1}(\varphi(Z_{X_3})) \wequi Z_{X_3} \amalg Z'$.
As before $z \not\in Z'$ and so $Z' \cap W_z^h = \emptyset$.
Thus $W_z^h \subset X_4 := X_3 \setminus Z'$.
By construction, $X_4 \to \A^\sigma \times W_z^h$ is an étale neighborhood of $Z_{X_4}$.
The composite $W_z^h \to X_4 \to \A^\sigma \times W_z^h$ consists of finitiely presented $W_z^h$-schemes.
Writing $W_z^h$ as the cofiltered limit of étale neighborhoods of $z$ in $W$, we conclude by continuity that there exists such an étale neighborhood $\tilde W \to W$ together with a model $\tilde W \to \tilde X_4 \to \A^\sigma \times \tilde W$ of the previous composite, still satisfying all the same properties (in particular $\tilde X_4 \to \A^\sigma \times \tilde W$ is an étale neighborhood of $Z_{\tilde X_4}$, and $\tilde X_4$ is an open subscheme of $X_1' \times_W \tilde W$).
Since $\tilde W \to W$ and $X_1' \to X$ are étale neighborhoods of $z$, so is $\tilde X_4 \to X$.

This concludes the case where $z \in X^{C_2}$.

\textbf{Case $z \not\in X^{C_2}$.}
Replacing $X$ by $X \setminus X^{C_2}$, we may assume that the action of $C_2$ is free.

\textbf{Trivial subcase.}
Suppose there is a smooth map $X \to B$, where $B$ is a smooth $0$-dimensional $k$-scheme with a free action by $C_2$.
In this case everything happening equivariantly over $B$ is equivalent to things happening non-equivariantly over $B/C_2$\NB{ref}, and so the result follows by appeal to the usual Gabber lemma.

\textbf{Reduction to $z$ set-theoretically fixed.}
Suppose that $\sigma z \ne z$.
Then the map $C_2 \times X \to X$ is an equivariant étale neighborhood of $z$ (since $C_2 z \wequi C_2 \times z$).
We may thus replace $X$ by $C_2 \times X$, and consequently assume that $X$ has a smooth map to $C_2$.
This is handled by the trivial subcase.
Consequently from now on we may assume that $z$ is set-theoretically fixed.

\textbf{Reduction to $\dim X > 1$.}
Existence of $z \in Z$ precludes $\dim X = 0$.
Hence we may assume $\dim X = 1$, so $z \in X$ is a closed point.
Since $C_2$ acts non-trivially on $k(z)$, the separable degree of $k(z)/k$ is positive (this follows from \cite[Tag 09HK]{stacks-project}\NB{$[k(z):k]_s = [k(z):k(z)^{C_2}]_s [k(z)^{C_2}:k]_s > 1$}).
Let $k_s \subset k(z)$ denote the maximal separable subextension, and set $B = Spec(k_s)$.
(This has a canonical $C_2$-action.)
Since $B \to B \times B$ is a clopen immersion, $X_B \to X$ contains an equivariant étale neighborhood of $z$.
If the action of $C_2$ on $B$ is non-trivial, we are done by the trivial subcase.
Otherwise we can repeat the construction, with the base field replaced by $k_s$.
Note that this construction has decreased the degree of $k(z)/k$, whence the process must terminate.

\textbf{Step 0:}
To summarize, we have $X$ with a free action by $C_2$, $z \in Z \subset X$ with $C_2 z = \{z\}$, $\dim X = n+1$, $n > 0$.
Let $\bar z$ be a closed specialization of $z$.

\textbf{Step 1:}
We claim that there is a map $\varphi: X \to \A^{n\sigma}$ which is smooth at $\bar z$, its restriction to $Z$ is quasi-finite at $\bar z$, and has $\varphi(\bar z) \ne 0$.
After choosing an open embedding of $X$ into a $C_2$-representation and considering linear projections, all three conditions are open.
Hence we need only prove that they can be satisfied.
By adding further generators to the embedding, we in fact just need to construct functions $X \to \A^{n\sigma}$ satisfying the smoothness/quasifiniteness/non-vanishing condition.

For the non-vanishing, we can use Lemma \ref{lemm:functionsC2}(2).

For the quasi-finiteness, it will be enough to prove that given $Z' \subset X$ closed of dimension $\le d$ ($d \ge 1$) with $\bar z \in Z'$, there exists $f: X \to \A^\sigma$ vanishing at $\bar z$ with $Z' \cap Z(f)$ of dimension $<d$.
In every component of $Z'$ of dimension $d$ we can find a closed point not conjugate to $\bar z$; call the resulting set $S$.
We thus need $f(\bar z) = 0$ and $f(S) \not\ni 0$.
This we find by Lemma \ref{lemm:functionsC2}(3).

For the smoothness, first choose finitely many generators of $m_{\bar z}$, and write $V$ for the affine space spanned by them.
We claim that if $f_1, \dots, f_n$ are general elements of $V$, then $\varphi = (f_1 - \sigma f_1, \dots, f_n - \sigma f_n)$ has the desired property.
We may verify this after base change to the algebraic closure.
Then $\bar z$ splits into finitely many free $C_2$ orbits $C_2 z_1, \dots, C_2 z_r$.
It will suffice to show that a general element is smooth at $C_2 z_1$.
The condition being open, we need only exhibit one specific set of maps $f_1, \dots, f_n$ satisfying the smoothness condition at $C_2 z_1$.
The map $V \to \Omega_{z_1} X \oplus \Omega_{\sigma z_1} X$ is surjective (the right hand side being a quotient of $(m_{\bar z}/m_{\bar z}^2) \otimes_k \bar k$).
Let $e_1, \dots, e_{n+1}$ be a basis of $\Omega_{z_1} X$.
Pick $f_i$ lifting $(e_i,0)$ to get the desired result (note that the action of $\sigma$ interchanges $(e_i, 0)$ and $(0, e_i)$) using \cite[17.11.1]{EGAIV}.

\textbf{Step 2:}
Shrinking $X$, we may thus assume given $\varphi: X \to \A^{n\sigma}$ which is smooth, and quasi-finite on $Z$, with $w := \varphi(z) \ne 0$.
Note that $w$ is set-theoretically fixed but has non-trivial action.
Let $W = (\A^{n\sigma})_w^h$ and $X_1 = X \times_{\A^{n\sigma}} W$.
Note that $z$ lifts to $X_1$.
Also $Z_{X_1} \to W$ is quasi-finite and hence as before we can obtain $X_2$ by removing finitely many clopen subsets of $Z_{X_1}$ in such a way that $z \in X_2$ and $Z_{X_2} \to W$ is finite; in fact $Z_{X_2}$ is local with closed point $z$ and so $z$ is the only point in $Z_{X_1}$ above $w$ (use incomparability in finite extensions \cite[Tag 00GT]{stacks-project}).
We claim that there is a map $X \to \A^1_W$ (trivial action on $\A^1$) such that the special fiber $X_w \to \A^1_w$ is quasi-finite at $z$, étale at $z$ and geometrically injective at $z$.
Assuming this, the map is flat at $z$ by miracle flatness \cite[Tag 00R4]{stacks-project}, hence étale at $z$ \cite[Tag 02GU]{stacks-project}.
Let $X_3$ be obtained by shrinking $X_2$ such that the map becomes étale.
Since $Z_{X_2}$ is local, we must have $Z_{X_2} \subset X_3$ and so $Z_{X_3} = Z_{X_2} \to W$ is still finite.
It follows as in Step 4 before (i.e. using Nakayama's lemma and \cite[Tag 04XV]{stacks-project}) that $Z_{X_3} \to \A^1_W$ is a closed immersion.

\textbf{Proof of claim:}
Suppose $f_w: X_w \to \A^1_w$ has the desired properties.
Let $f: X \to \A^1_W$ be any (non-equivariant) lift of $f_w$ (this exists because $f$ corresponds to an element of $\scr O(X)$ with specfied image in $\scr O(X_w)$, and $X_w \to X$ is a closed immersion of affine schemes).
The antisymmetrization of $f$ will have the desired properties.
We thus reduce to the case $W=w$.
We are thus given $X$ smooth of dimension $1$ over $w$ with a specified closed point $z$ and must construct a map to $\A^1_w$ which is quasi-finite, étale and geometrically injective at $z$.
Since $w$ has a free action, by descent we can instead consider the non-equivariant situation over $w/C_2$.
Thus given a smooth affine curve $X$ over a field $v$ with a closed point $z$, we must construct a map $X \to \A^1$ which is quasi-finite, étale and geometrically injective at $z$.
Choosing an embedding $X \hookrightarrow \A^N_v$ and considering linear projections to $\A^1$, all three conditions become open.
Being quasi-finite just requires being non-constant (at $z$), which can clearly be satisfied.
The universal injectivity can be satisfied by Lemma \ref{lemm:functions}.
Generical étaleness is also well-known using arguments as in smoothness for step 1; see e.g. \cite[\S3.2.2]{kai-moving-lemma}.

\textbf{Step 3:}
As in step 5 from case 1, we remove $\varphi^{-1}(\varphi(Z_{X_3})) \setminus Z_{X_3}$ from $X_3$ in order to obtain an étale neighborhood $X_4 \to \A^1_W$ of $Z_{X_4} = Z_{X_3}$.
These are finitely presented schemes over $W$, which itself is the cofiltered limit of étale neighborhoods of $w \in \A^{n\sigma}$.
It follows that there exists an étale neighborhood $\tilde W \to \A^{n\sigma}$ of $w$ together with a model $\tilde X_4 \to \A^1_{\tilde W}$ having the same properties as before.
Since $\tilde W \to \A^{n\sigma}$ is an étale neighborhood of $W$, $\tilde X_4 \to X$ is an étale neighborhood of $z$.

This concludes the proof.
\end{proof}

\section{Applications}
Throughout let $k$ be an infinite field of characteristic not $2$.
Denote by \[ \SH^{C_2,S^1}(k) \subset \Fun((\Sm^{C_2}_k)^\op, \SH) \] the category of $\A^1$-invariant Nisnevich sheaves of spectra.

\begin{lemma} \label{lemm:inj}
Let $E \in \SH^{C_2,S^1}(k)$, $X \in \Sm^{C_2}_k$ and $x \in X$.
Write $S$ for the generic orbit of $X_{C_2 x}^h$ and $S'$ for the generic points of $(X_{C_2 x}^h)^{C_2}$ (of which there are $0$ or $1$).
The map \[ \pi_0E(X_{C_2 x}^h) \to \pi_0 E(S) \oplus \pi_0 E(S') \] is injective.
\end{lemma}
\begin{proof}
We prove the result by induction on $\dim X$.
The case of dimension $0$ is tautological.
Set $Y=X_{C_2 x}^h$.
Let $a \in \pi_0 E(Y)$ with $a|_{S} = 0$ and $a|_{S'} = 0$.
Since the action on $X^{C_2}$ is trivial, $\pi_0 E(Y^{C_2}) \to \pi_0 E(S')$ is injective (by non-equivariant Gersten injectivity; see e.g. \cite[Corollary 6.2.4]{bloch-ogus-gabber}); hence $a|_{Y^{C_2}} = 0$.
By continuity we find an étale neighborhood $X_1 \to X$ of $C_2 x$, a closed subscheme $Z \subset X_1$ and a class $b \in \pi_0E(X_1)$ such that $b|_Y = a$, $b|_{X_1 \setminus Z} =0$ and $b|_{X_1^{C_2}}=0$.
Applying Theorem \ref{thm:main} we obtain a further étale neighborhood $X_2 \to X_1$ of $C_2x$, a smooth $C_2$-scheme $W$ and an étale neighborhood $X_2 \to \A^1_W$ of $Z_{X_2}$.
Moreover if the action on $\A^1$ is non-trivial, we obtain a lift of $W$ into $X_2$.
We pull back the situation to the henselization of $W$ in the image of $C_2 x$; hence we may assume $W$ henselian in the image of $C_2 x$.
Set $Z_2 = Z_{X_2}$ and $F=\pi_0 E$.
From this we obtain the following commutative diagram with exact rows
\begin{equation*}
\begin{CD}
@.                     F(W) \\
    @.                  @A?AA \\
F(X_2 \setminus Z_2) @<<< F(X_2) @<<< F(X_2/X_2 \setminus Z_2) \\
@AAA                   @AAA           @A{\wequi}AA \\
F(\A^1_W \setminus Z_2) @<<< F(\A^1_W) @<<< F(\A^1_W/\A^1_W \setminus Z_2).
\end{CD}
\end{equation*}
The map labelled $?$ only exists if the action on $\A^1$ is non-trivial.
From the diagram we conclude that there is a class $c \in F(\A^1_W)$ with $c|_{X_2} = b|_{X_2}$ (so $c|_Y = a$) and $c|_{\A^1_W \setminus Z_2} =0$.
It will suffice to show that $c=0$.

First consider the case where the action on $\A^1$ is trivial.
In this case the map $\A^1_W \setminus Z_2 \to W$ has a section.
(To see this, note that an invariant function on $W$ defines a section missing $Z_2$ if and only if this is true in the special fibers ($Z_2$ being finite over $W$).
Since invariant functions can be lifted along closed immersions, to prove the claim we may replace $w$ by its special orbit (the image of $C_2 x$).
This special orbit has either trivial or free action, and both cases are easily dealt with.)
This implies that $F(W) \wequi F(\A^1_W) \to F(\A^1_W \setminus Z_2)$ is injective and hence $c=0$, as needed.

Now we consider the case where the action on $\A^1$ is non-trivial; in particular we are given a lift of $W$ into $X_2$ and the map $?$ exists.
By assumption $b|_{X_2^{C_2}} = 0$ and hence also $c|_{W^{C_2}} = 0$.
If the action of $C_2$ on $W$ is trivial then $c=0$ (since $F(\A^1_W) \wequi F(W)$).
Otherwise the action of $C_2$ on the generic orbit $C_2 \eta \in W$ must be non-trivial, so free.
Since $Z_2 \to W$ is finite, the map $\A^1_{C_2\eta} \setminus Z_2 \to C_2 \eta$ has a section.
(Since $C_2 \eta$ has free action, this problem corresponds to one over $(C_2 \eta)/\eta$, and even though $\A^1$ has the sign action, the quotient corresponds to a line bundle on $(C_2\eta)/\eta$ which must be trivial.)
This implies that $c|_{C_2 \eta} = 0$, and hence $c=0$ by induction.\NB{$\dim W < \dim X$!}
\end{proof}

We arrive at the main point.
Write $\Shv_\Nis(\Sm^{C_2}_k, \SH)$ for the category of spectral Nisnevich sheaves on $\Sm_k^{C_2}$.
Recall that, as the stabilization of an $\infty$-topos, this has a standard $t$-structure with heart the sheaves of abelian groups on $\Sm_k^{C_2}$ \cite[Proposition 1.3.2.1]{lurie-sag}.
In fact for $E \in \Shv_\Nis(\Sm^{C_2}_k, \SH)$ set $\ul\pi_i E = a_\Nis \pi_i E$.
Then the non-negative/non-positive spectral presheaves are detected by the evident vanishing of homotopy sheaves $\ul\pi_i$, and the functor $\ul\pi_0$ is an equivalence when restricted to the heart.
We call a Nisnevich sheaf of abelian groups on $\Sm_k^{C_2}$ \emph{strictly $\A^1$-invariant} if all of its cohomology is $\A^1$-invariant; equivalently, the associated Eilenberg--MacLane spectrum is $\A^1$-invariant.
\begin{proposition}
Let $k$ be an infinite field of characteristic $\ne 2$.
\begin{enumerate}
\item Let $E \in \SH^{C_2,S^1}(k)$.
  Then $\ul\pi_i E = 0$ if and only if, for every essentially smooth $C_2$-scheme $O$ over $k$ of dimension $0$ (i.e. a finite disjoint union of spectra of finitely generated separable field extensions of $k$, with some action by $C_2$), we have $\ul\pi_i E(O) = 0$.
\item Let $E \in \Shv_\Nis(\Sm^{C_2}_k, \SH)_{\ge 0}$.
  Then $L_\mot E \in \Shv_\Nis(\Sm^{C_2}_k, \SH)_{\ge 0}$.
\item The standard $t$-structure on $\Shv_\Nis(\Sm^{C_2}_k, \SH)_{\ge 0}$ restricts to the subcategory $\SH^{C_2,S^1}(k)$.
  The heart is the category of strictly $\A^1$-invariant sheaves on $\Sm_k^{C_2}$.
\end{enumerate}
\end{proposition}
\begin{proof}
(1) Necessity is clear, and sufficiency is immediate from Lemma \ref{lemm:inj}.

(2) We must prove that $(L_\mot E)(O)$ a connective spectrum for every $O$ as in (1).
Since $(L_\mot E)(X)$ is obtained\footnote{This singular construction is $\A^1$-equivalent to $E$ and $\A^1$-invariant \cite[\S2 Corollaries 3.5 and 3.8]{A1-homotopy-theory}, and it is a Nisnevich sheaf since the topology is defined by a cd structure; hence it is the motivic localization.} as the singular construction $|E(X \times \A^{\bullet})|$, for this it suffices to show that $E(O \times \A^n)$ is $(-n)$-connective.
This holds since $O \times \A^n$ has Nisnevich cohomological dimension $\le n$ \cite[Proposition A.4.4]{BKRS}.

(3) Let $E \in \SH^{C_2,S^1}(k)$ and write $\tau_{\ge 0}^\Nis E \in\Shv_\Nis(\Sm^{C_2}_k, \SH)_{\ge 0}$ for the connective cover.
The canonical map $\tau_{\ge 0}^\Nis E \to E$ factors through $L_\mot \tau_{\ge 0}^\Nis E$, $E$ being $\A^1$-invariant.
On the other hand, $L_\mot \tau_{\ge 0}^\Nis E$ being connective, the map $L_\mot \tau_{\ge 0}^\Nis E \to E$ factors through $\tau_{\ge 0}^\Nis E$.
We have now exhibited $\tau_{\ge 0}^\Nis E$ as a retract of the motivically local object $L_\mot \tau_{\ge 0}^\Nis E \to E$, whence $\tau_{\ge 0}^\Nis E \in \SH^{C_2,S^1}(k)$.
It follows that the standard $t$-structure restricts as claimed.
The identification of the heart follows by definition.
\end{proof}

\bibliographystyle{alpha}
\bibliography{bibliography}

\begin{thebibliography}{DHKY21}

\bibitem[BKRS22]{BKRS}
Tom Bachmann, Adeel~A. Khan, Charanya Ravi, and Vladimir Sosnilo.
\newblock Categorical {M}ilnor squares and {K}-theory of algebraic stacks.
\newblock {\em Selecta Math. (N.S.)}, 28(5):Paper No. 85, 72, 2022.

\bibitem[CTHK97]{bloch-ogus-gabber}
Jean-Louis Colliot-Th\'{e}l\`ene, Raymond~T. Hoobler, and Bruno Kahn.
\newblock The {B}loch-{O}gus-{G}abber theorem.
\newblock In {\em Algebraic {$K$}-theory ({T}oronto, {ON}, 1996)}, volume~16 of
  {\em Fields Inst. Commun.}, pages 31--94. Amer. Math. Soc., Providence, RI,
  1997.

\bibitem[DHKY21]{gabber-lemma-noetherian-domains}
Neeraj Deshmukh, Amit Hogadi, Girish Kulkarni, and Suraj Yadav.
\newblock Gabber's presentation lemma over noetherian domains.
\newblock {\em J. Algebra}, 569:169--179, 2021.

\bibitem[Gab94]{gabber-presentation-lemma}
Ofer Gabber.
\newblock Gersten's conjecture for some complexes of vanishing cycles.
\newblock {\em Manuscripta Math.}, 85(3-4):323--343, 1994.

\bibitem[{Gro}67]{EGAIV}
A.~{Grothendieck}.
\newblock {\'El\'ements de g\'eom\'etrie alg\'ebrique. IV: \'Etude locale des
  sch\'emas et des morphismes de sch\'emas. R\'edig\'e avec la colloboration de
  Jean Dieudonn\'e.}
\newblock {\em {Publ. Math., Inst. Hautes \'Etud. Sci.}}, 32:1--361, 1967.

\bibitem[HK20]{gabber-lemma-finite-fields}
Amit Hogadi and Girish Kulkarni.
\newblock Gabber's presentation lemma for finite fields.
\newblock {\em J. Reine Angew. Math.}, 759:265--289, 2020.

\bibitem[Hoy17]{hoyois-equivariant}
Marc Hoyois.
\newblock The six operations in equivariant motivic homotopy theory.
\newblock {\em Advances in Mathematics}, 305:197--279, 2017.

\bibitem[HRO18]{rigidity-equiv-pseudopretheories}
Jeremiah Heller, Charanya Ravi, and Paul~Arne \O{}stv\ae{}r.
\newblock Rigidity for equivariant pseudo pretheories.
\newblock {\em J. Algebra}, 516:373--395, 2018.

\bibitem[Kai21]{kai-moving-lemma}
Wataru Kai.
\newblock A moving lemma for algebraic cycles with modulus and contravariance.
\newblock {\em Int. Math. Res. Not. IMRN}, (1):475--522, 2021.

\bibitem[Lur18]{lurie-sag}
Jacob Lurie.
\newblock Spectral algebraic geometry.
\newblock February 2018.

\bibitem[Mor12]{A1-alg-top}
Fabien Morel.
\newblock {\em $\mathbb{A}^1$-Algebraic Topology over a Field}.
\newblock Lecture Notes in Mathematics. Springer Berlin Heidelberg, 2012.

\bibitem[MV99]{A1-homotopy-theory}
Fabien Morel and Vladimir Voevodsky.
\newblock $\mathbb{A}^1$-homotopy theory of schemes.
\newblock {\em Publications Mathématiques de l'Institut des Hautes Études
  Scientifiques}, 90(1):45--143, 1999.

\bibitem[SS18]{schmidt2018stable}
Johannes Schmidt and Florian Strunk.
\newblock Stable $\mathbb{A}^1$-connectivity over dedekind schemes.
\newblock {\em Annals of K-Theory}, 3(2):331--367, 2018.

\bibitem[{Sta}18]{stacks-project}
The {Stacks Project Authors}.
\newblock {\itshape Stacks Project}.
\newblock \url{http://stacks.math.columbia.edu}, 2018.

\end{thebibliography}

\end{document}